\documentclass[a4paper]{amsart}
\usepackage[utf8]{inputenc}
\usepackage{amssymb,amsmath,enumerate,amsthm,url}
\usepackage[all,cmtip]{xy}
\usepackage{graphicx}
\usepackage[numbers]{natbib}
\usepackage{bibentry}
\usepackage{hyperref}

\theoremstyle{plain}
\newtheorem{theorem}{Theorem}[section]
\newtheorem{lemma}[theorem]{Lemma}

\newtheorem{fact}[theorem]{Fact}
\newtheorem*{fact*}{Fact}
\newtheorem{corollary}[theorem]{Corollary}

\newtheorem*{claim*}{Claim}
\theoremstyle{definition}
\newtheorem{definition}[theorem]{Definition}
\newtheorem*{definition*}{Definition}

\newtheorem*{notation*}{Notation}
\theoremstyle{remark}
\newtheorem{remark}[theorem]{Remark}
\newtheorem*{remark*}{Remark}

\newtheorem*{example*}{Example}

\newtheorem*{note*}{Note}

\newtheorem*{question*}{Question}
\begin{document}
\newcommand{\Fix}{{\operatorname{Fix}}}
\newcommand{\Aut}{{\operatorname{Aut}}}
\newcommand{\Bir}{{\operatorname{Bir}}}
\newcommand{\tp}{{\operatorname{tp}}}
\newcommand{\acl}{{\operatorname{acl}}}
\newcommand{\dcl}{{\operatorname{dcl}}}
\newcommand{\ACF}{{\operatorname{ACF}}}
\newcommand{\AGL}{{\operatorname{AGL}}}
\newcommand{\GL}{{\operatorname{GL}}}
\newcommand{\id}{{\operatorname{id}}}
\newcommand{\trd}{{\operatorname{trd}}}
\newcommand{\locus}{{\operatorname{locus}}}
\renewcommand{\P}{{\mathbb{P}}}
\newcommand{\A}{{\mathbb{A}}}
\newcommand{\Q}{{\mathbb{Q}}}
\newcommand{\G}{{\mathbb{G}}}
\newcommand{\C}{{\mathbb{C}}}
\newcommand{\N}{{\mathbb{N}}}
\newcommand{\U}{{\mathcal{U}}}
\renewcommand{\a}{{\overline{a}}}
\providecommand{\defn}[1]{{\bf #1}}
\providecommand{\bdl}{\boldsymbol\delta}

\providecommand{\dind}{\ind^{\bdl}}
\providecommand{\ndind}{\nind^{\bdl}}

\providecommand{\gfrac}[2]{#2^{-1}\cdot #1}

\def\Ind#1#2{#1\setbox0=\hbox{$#1x$}\kern\wd0\hbox to 0pt{\hss$#1\mid$\hss}
\lower.9\ht0\hbox to 0pt{\hss$#1\smile$\hss}\kern\wd0}
\def\ind{\mathop{\mathpalette\Ind\emptyset }}
\def\notind#1#2{#1\setbox0=\hbox{$#1x$}\kern\wd0\hbox to 0pt{\mathchardef
\nn=12854\hss$#1\nn$\kern1.4\wd0\hss}\hbox to
0pt{\hss$#1\mid$\hss}\lower.9\ht0 \hbox to
0pt{\hss$#1\smile$\hss}\kern\wd0}
\def\nind{\mathop{\mathpalette\notind\emptyset }}

\newbox\gnBoxA
\newdimen\gnCornerHgt
\setbox\gnBoxA=\hbox{$\ulcorner$}
\global\gnCornerHgt=\ht\gnBoxA
\newdimen\gnArgHgt
\def\code #1{%
\setbox\gnBoxA=\hbox{$#1$}%
\gnArgHgt=\ht\gnBoxA%
\ifnum     \gnArgHgt<\gnCornerHgt \gnArgHgt=0pt%
\else \advance \gnArgHgt by -\gnCornerHgt%
\fi \raise\gnArgHgt\hbox{$\ulcorner$} \box\gnBoxA %
\raise\gnArgHgt\hbox{$\urcorner$}}

\title{Non-expansion in polynomial automorphisms of $\C^2$}
\author{Martin Bays and Tingxiang Zou}
\address{Martin Bays, Mathematical Institute, University of Oxford, Andrew Wiles Building, Radcliffe Observatory Quarter, Woodstock Road, Oxford OX2 6GG, UK}
\email{mbays@sdf.org}
\address{Tingxiang Zou, Mathematical Institute, University of Bonn, Endenicher Allee 60, 53115 Bonn, Germany}
\email{tzou@math.uni-bonn.de}

\maketitle

\begin{abstract}
    We treat the higher-dimensional Elekes-Szabó problem in the case of the action of $\Aut(\C^2)$ on $\C^2$.
\end{abstract}

\section{Introduction}
Consider the group $\Aut(\C^2)$ of polynomial automorphisms of the complex plane. An element of $\Aut(\C^2)$ can be seen as a pair of polynomials in two variables $g_x,g_y \in \C[x,y]$, acting on $\C^2$ by $(g_x,g_y)*(x,y) := (g_x(x,y),g_y(x,y))$. We consider the question: if $D \subseteq \Aut(\C^2)$ is a finite set of polynomial automorphisms of bounded degree, and $A \subseteq \C^2$ is a finite subset of the plane, how can we have non-expansion of the form $|D*A| = \{g*a : (g,a) \in D\times A\} < |A|^{1+\eta}$? This has the form of an Elekes-Szabó problem \cite{ES-groups}, but since $\C^2$ is 2-dimensional it falls outside of the scope of previous results. In \cite[Question~7.5]{homogES}, we proposed a tentative solution to such higher-dimensional Elekes-Szabó problems, which in this case would imply that, for $A$ not concentrating on a curve and of size comparable in exponent to $|D|$, non-expansion is possible only when $D$ concentrates on a coset of a nilpotent algebraic subgroup.

In this paper, we use the amalgamated product structure of $\Aut(\C^2)$, along with model-theoretic techniques originating in \cite{Hr-psfDims} and further developed in \cite{cubicSurfaces} and \cite{homogES}, to confirm this; see Theorem~\ref{c:main-fin} for a precise statement of this form. Analysing the nilpotent subgroups which arise, we furthermore obtain the following statement in the style of the Elekes-Rónyai analysis of expanding bivariate polynomials. The original Elekes-Rónyai result \cite{ER} showed that non-expanding bivariate polynomials are conjugate to addition or multiplication; the analogous condition in the present situation is a little more complicated, and we term it \emph{co-ordinate separability} because it arises from constrained interaction between the variables.

\begin{definition}
    Let $F(x,y,\bar{z}),G(x,y,\bar{z})$ be polynomials over $\C$. We say the pair $(F,G)$ is \emph{co-ordinate separable} if there are $g,h\in\Aut(\C^2)$ and polynomials $t_0(\bar z)$, $t_1(\bar z)$ and $s(y,\bar z)$ such that $g\circ (F(x,y,\bar z),G(x,y,\bar z)) \circ h$ is one of the following:
    \begin{itemize}
        \item $(t_0(\bar z)x,~t_1(\bar z)y)$;
        \item $(t_0(\bar z)x,~y+t_1(\bar z))$;
        \item $(t_1(\bar z)^\ell x+t_0(\bar z)y^\ell,~ t_1(\bar z)y)$ for some $\ell\in\N$;
        \item $(x+s(y,\bar z),~y+t_1(\bar z))$.
    \end{itemize}
\end{definition}

\begin{theorem}\label{t:ER-main}
    Let $F(x,y,\bar{z}),G(x,y,\bar{z})\in\C[x,y,\bar z]\setminus \C[x,y]$. Suppose $(F,G)$ is not co-ordinate separable. Then for any $\varepsilon > 0$, there is $\eta>0$ such that for all finite sets $A\subseteq \C^2$ and $B\subseteq \C^{|\bar z|}$ with
    \begin{enumerate}[(i)]
        \item $|A|^{1/\varepsilon}\geq |B|\geq |A|^\varepsilon\geq \frac1\eta$,
        \item $(F(x,y,b),G(x,y,b))\in\Aut(\C^2)$ for all $b\in B$,
        \item $|A\cap V|\leq |A|^{1-\varepsilon}$ and $|B\cap W|\leq |B|^{1-\varepsilon}$ for all varieties $V\subsetneq \C^2$ and $W\subsetneq \C^{|\bar z|}$ of complexity $<\frac1\eta$,
    \end{enumerate}
    we have $|(F,G)(B)*A| := |\{(F(a,b),G(a,b)):a\in A,b\in B\}|\geq |A|^{1+\eta}$.
\end{theorem}

Note that assumption (iii), which we call ``weak general position'', is achieved in the case of grids, i.e.\ when $A=A_0\times A_1$ with $|A_i|\geq |A|^\varepsilon$ and similarly for $B$.
In the case that the assumption on $A$ fails, meaning that $A$ concentrates on some algebraic curve, the results in \cite{homogES} already apply to yield that abelian groups explain non-expansion.

Let us also remark that the fourth case in the definition of co-ordinate separability is the most interesting, because it corresponds to a nilpotent algebraic group, of nilpotency class depending on the degree in $y$ of $s$. This appears to be the first time that nilpotent non-abelian groups have appeared explicitly in such an Elekes-Szabó result beyond the foundational case \cite{BGT-lin} of approximate subgroups of complex algebraic groups, and we take it as evidence towards our expectation expressed in \cite[Question~7.5]{homogES} that nilpotent algebraic group actions are precisely the structures behind (ternary) Elekes-Szabó phenomena in general. In the present case, the action is already apparent, and the main difficulty is to show that an algebraic subgroup of $\Aut(\C^2)$ is responsible -- $\Aut(\C^2)$ itself being infinite dimensional and not an algebraic group. This suggests considering the corresponding question in the generality of an action on a variety of a group defined as a directed limit of constructible sets, for example the action on $\P^2(\C)$ of the Cremona group $\Bir(\C^2)$ of birational automorphisms of the plane. However, the techniques of this paper crucially exploit the group structure of $\Aut(\C^2)$ (see Remark~\ref{r:genAmalg}), and do not apply even to this case of $\Bir(\C^2)$. We aim to treat this and more general situations in future work.

\section{Group structure of $\Aut(k^2)$}
Let $k$ be an algebraically closed field of characteristic 0.

Consider the group $G := \Aut(k^2)$ of polynomial automorphisms of $k^2$, consisting of those $(g_x,g_y) \in k[x,y]^2$ admitting a compositional inverse which is also of this form.
Jung's Theorem describes $G$ as an amalgamated product, as follows.
Let
\begin{align*}
  E&:=\{(x,y)\mapsto(ax+P(y),by+c):
  P(y)\in k[y],\; a,b,c\in k,\;ab\neq 0\},\\
  A&:=\{(x,y)\mapsto(a_1x+b_1y+c_1,a_2x+b_2y+c_2):
  a_i,b_i,c_i\in k,\; a_1b_2-a_2b_1\neq 0\},\\
  S&:= A\cap E.
\end{align*}
Here $E$ is the group of \emph{elementary automorphisms}, and $A$ is the group of \emph{affine automorphisms}. It is easy to see that $E$ is a countable union of algebraic subgroups $(E_n)_{n\geq 0}$, where $E_n$ is the collection of elements in $E$ in which we restrict the polynomial $P(y)$ to have degree $\leq n$.
Note that $S = E_1$.

Jung's Theorem is then:
\begin{fact}[{\cite{jung}}]
    $G$ is an amalgamated free product of $E$ and $A$ over $S$.
\end{fact}

We proceed to deduce some useful consequences of this fact.
We use the following description of the elements of an amalgamated free 
product, for which see \cite[I.1, Remark to Theorem 1]{Serre}.


\begin{fact} \label{f:altWords}
    Each $g \in G \setminus S$ can be written as an \defn{alternating word},
    $g = \prod_{i < n} h_i$ where each $h_i \in (A \cup E) \setminus S$ and $h_i \in A 
    \Leftrightarrow  h_{i+1} \in E$.
    Furthermore,
    if $g = \prod_{i < n'} h'_i$ is another such expression, then $n'=n$ and there exist
    $s_0,\ldots ,s_{n-2} \in S^{n-1}$ such that
    $(h'_0,\ldots ,h'_{n-1}) = (h_0\cdot s_0^{-1},s_0\cdot h_1\cdot s_1^{-1},\ldots ,s_{n-2}\cdot h_{n-1})$.

    In particular,
    $\prod_{i<k} h_i \cdot S = \prod_{i<k}h'_i \cdot S$
    and $S\cdot\prod_{i\geq k}h_i = S \cdot \prod_{i\geq k}h'_i$
    for any $k \leq  n$.
\end{fact}

\begin{lemma} \label{l:algPrefix}
    Suppose $g = \prod_{i<n} h_i$ is an alternating word, and $k\leq n$.
    Then the variety $\prod_{i<k} h_i \cdot S$ is defined over the field $\Q(g)$.
\end{lemma}
\begin{proof}
    If $\sigma$ is a field automorphism of $k$ fixing $g$,
    then $g = \prod_{i<n} \sigma(h_i)$ is another alternating word,
    and so by Fact~\ref{f:altWords} we have
    \[\sigma\left({\prod_{i<k} h_i \cdot S}\right) =
    \prod_{i<k} \sigma(h_i) \cdot S = \prod_{i<k} h_i \cdot S .\qedhere\]
\end{proof}

\begin{lemma} \label{l:algLetters}
    Any $g \in G \setminus S$ can be written as an alternating word $g = \prod_{i<n} 
    h_i$ with each $h_i$ algebraic over $g$.
\end{lemma}
\begin{proof}
    Apply Fact~\ref{f:altWords} with the field $k$ being the algebraic closure of $g$.
\end{proof}

We use conjugation notation $h^g := g^{-1}hg$.
\begin{lemma} \label{l:conjugate}
    Let $f=\prod_{i\leq n}f_i$ and $g=\prod_{i\leq n}g_i$ be alternating words 
    such that $f_i \in A \Leftrightarrow g_i \in A$ for all $i$.
    Suppose $\gfrac gf$ is conjugate in $G$ to an element of $A\cup E$.
    Then there is $m\leq n$ and $\beta \in A\cup E$ such that
    $\gfrac gf = \beta^{(\prod_{i>m}g_i)}$.
\end{lemma}
\begin{proof}
    Let $m \leq  n$ be greatest such that
    $(\prod_{i<m}f_i)\cdot S=(\prod_{i<m}g_i)\cdot S$.
    If $m=n$ then we are done with $\beta = \gfrac gf \in S$,
    so suppose $m<n$.
    Then setting $\gamma := \gfrac{\prod_{i\leq m}g_i}{(\prod_{i\leq m}f_i)}$,
    we have $\gamma \notin S$,
    and $\gamma = f_m^{-1} \cdot \gfrac{\prod_{i<m}g_i}{(\prod_{i<m}f_i)}\cdot g_m \in (f_m^{-1}\cdot S\cdot g_m)$ so $\gamma \in A \Leftrightarrow  g_m \in A$,
    and so
    \begin{equation} \label{e:altW}
        \gfrac gf = (\prod_{i>m}f_i)^{-1}\cdot\gamma\cdot\prod_{i>m}g_i,
    \end{equation}
    is an alternating word.

    Now say $\gfrac gf = \alpha^h$ where $\alpha \in A\cup E$.
    We may assume that either $h = 1$ or $h = \prod_{i<k}h_i$
    is an alternating word with $\alpha^{h_0} \notin A\cup E$,
    since otherwise we may replace $\alpha$ with $\alpha^{h_0}$.
    
    In the case $h=1$, we conclude with $\beta := \alpha$ and $m := n$, so suppose $h \neq  1$.
    Then $\alpha \notin S$,
    and $\alpha^h=(\prod_{i<k}h_i)^{-1}\cdot \alpha\cdot\prod_{i<k}h_i$ is an alternating word.
    Comparing with our previous alternating word \eqref{e:altW},
    we obtain $S\cdot\prod_{i>m}g_i = S\cdot\prod_{i<k}h_i$,
    say $\prod_{i<k}h_i = s\cdot\prod_{i>m}g_i$,
    and we conclude with $\beta := \alpha^s$.
\end{proof}

We also use the following fact that an element of $G$ with infinite fixed point 
set must be conjugate to an elementary automorphism.
\begin{fact}[{\cite[Theorem~2.1, Theorem~0.1]{Bedford-autC2}}]\label{f:infFix}
    Suppose $g \in G$ is such that $g*x=x$ for infinitely many $x \in k^2$. 
    Then there exists $h \in G$ such that $g^h \in E$.
\end{fact}

\subsection{Nilpotent algebraic subgroups}

\begin{lemma}\label{l:nilS}
  Any nilpotent algebraic subgroup $H \leq A$ is conjugate in $A$ to a 
  subgroup of $S = E_1$.
\end{lemma}
\begin{proof}
  First observe that $A \cong  \G_a^2 \rtimes \GL_2$.

  Let $\pi : A \rightarrow  \GL_2$ be the corresponding projection homomorphism.
  Then $\pi(H)$ is a nilpotent algebraic subgroup of $\GL_2$, so in particular 
  solvable, hence $\pi(H)$ is conjugate in $\GL_2$ to a subgroup of the Borel 
  subgroup of upper triangular matrices $T \leq  \GL_2$.
  Say $\pi(H)^\alpha \leq  T$ with $\alpha \in \GL_2$.

  Let $\beta \in \pi^{-1}(\alpha) \subseteq  A$.
  Then $\pi(H^\beta) = \pi(H)^\alpha \leq  T$,
  so $H^\beta \leq  \pi^{-1}(T) = S$.
\end{proof}

\begin{lemma}\label{l:nilE}
    Fix $n\in\N$. Any connected nilpotent algebraic subgroup of $E_n$ over $k$ is conjugate (in $E_n$) to a subgroup of one of the following nilpotent algebraic groups:
    \begin{enumerate}
       \item $\{(x,y)\mapsto (ax,by): a,b\in k^*\}$.
       \item $\{(x,y)\mapsto (ax,y+b): a\in k^*,b\in k\}$.
       \item $\{(x,y)\mapsto (b^\ell x+ay^\ell,by): a\in k, b\in k^*\}$ for some $\ell\in\N$.
       \item $\{(x,y)\mapsto (x+P(y),y+b): P(y)\in k[y]_{\leq n},b\in k\}$.
    \end{enumerate}
    The groups in (1)-(3) are abelian.
\end{lemma}
\begin{proof}
  \newcommand{\ee}[1]{\langle#1\rangle}
  In this proof, we write elements of $E$ in the form $\ee{ax+P(y),by+c}$.

  Let $N$ be a nilpotent subgroup of $E_n$. 

  We consider the following algebraic subgroups of $E_n$:
  \begin{align*}
  G_{x,a} &:= \{\ee{x+P(y),y} : P\in k[x]_{\leq n}\} \\
  G_{x,m} &:= \{\ee{ax,y} : a \in k^*\} \cong \G_m\\
  G_{y,a} &:= \{\ee{x,y+c} : c\in k\} \cong \G_a\\
  G_{y,m} &:= \{\ee{x,by} : b\in k^*\} \cong \G_m\\
  G_x &:= \{\ee{ax+P(y),y}:P\in k[x]_{\leq n},a\in k^*\} \cong G_{x,a}\rtimes G_{x,m} \\
  G_y &:= \{\ee{x, by+c}:b\in k^*, c\in k\} \cong G_{y,a}\rtimes G_{y,m}
  \end{align*}
  It is easy to see that $G_x$ is a normal subgroup of $E_n$ and $E_n=G_xG_y$, and thus $E_n\cong G_x\rtimes G_y$. Consider the corresponding group homomorphism
  \[\pi:G\to G_y;\; \ee{ax+P(y),by+c} \mapsto \ee{x,by+c}.\]
  Then $\pi(N)$ is a nilpotent subgroup of $G_y \cong G_{y,a}\rtimes G_{y,m}$, so by conjugating we may assume that $\pi(N) \leq  G_{y,a}$ or $\pi(N) \leq  G_{y,m}$.

Now consider $H:= G_x\cap N$, a nilpotent subgroup of $G_x$. 
For elements of $G_x$, we denote $\ee{ax+P(y),y}$ by $ax+P(y)$.

Suppose there $H \not\leq G_{x,a}$, so say $\phi = ax+P(y)\in H$ with $a\neq 1$. Let $\psi = a'x+Q(y) \in G_x$. Then we calculate
\[ [\phi,\psi]=x+(aa')^{-1}((a-1)Q(y)-(a'-1)P(y)). \] 
Thus, $C_{G_x}(\phi)=\{a'x+(a-1)^{-1}(a'-1)P(y):a'\in k^*\}=(G_{x,m})^\theta$ where $\theta:=(a-1)x+P(y)$. Suppose $H\not\leq (G_{x,m})^\theta$. Then there is $\phi'=x+P'(y)\in [H,H]$ with $P'(y)\neq 0$. Consider $[\phi',\phi]=x+a^{-1}(1-a)P'(y)\neq x$. Inductively, $[\cdots [[\phi',\phi],\phi],\cdots]\neq x$, contradicting $H$ being nilpotent. We conclude that, by conjugating $N$ by an element of $G_x$, which does not change $\pi(N)$, we may assume that either $H \leq  G_{x,a}$ or $H \leq  G_{x,m}$. 

Suppose $H$ is a non-trivial subgroup of $G_{x,m}$, so say $\psi = \ee{a'x,y}\in H$ with $a'\neq 1$.
Now $H$ is normal in $N$, so if $\phi = \ee{ax+P(y),cy+d} \in N$, then
\[\psi^\phi=\ee{a'x+a^{-1}(a'-1)P(y), y} \in G_{x,m},\] and so $P(y)=0$. So $N$ is a nilpotent subgroup of $\{\ee{ax,by+c}: a,b\in k^*,c\in k\}$, which is isomorphic to the upper triangular subgroup of $\GL_2$, and so we see that $N$ is conjugate to a subgroup of 
\begin{itemize}
  \item $\{\ee{ax,by}: a,b\in k^*\}\cong \G_m^2$ or
  \item $\{\ee{ax,y+b}: a\in k^*,b\in k\}\cong \G_m\times \G_a$.
\end{itemize}

Otherwise, $H\leq G_{x,a}$. Then $N$ induces an algebraic homomorphism $\rho : \pi(N) \to G_{x,m}$ defined by
$\rho(c) = a$ if $\ee{ax+P(y),y+c} \in N$, which is well-defined since $N \cap G_x = H\leq G_{x,a}$.

First suppose $\pi(N)\leq G_{y,a}$. Then $\rho$ has trivial image, since there are no non-trivial algebraic homomorphisms $\G_a \to \G_m$, and so \[N\leq \{\ee{x+P(y),y+b}: P(y)\in k[y]_{\leq n},b\in k\},\] as required.

We are left with the case that $H\leq G_{x,a}$ and $\pi(N)$ is a non-trivial subgroup of $G_{y,m}$, in which case $\pi(N) = G_{y,m}$ by connectedness of $N$. Then $\rho : G_{y,m}\to G_{x,m}$ induces a homomorphism from $\G_m\to\G_m$, and so \[N\leq \{\ee{b^\ell x+P(y),by}: P(y)\in k[y]_{\leq n},b\in k^*\}:=G_\ell,\] for some constant $\ell\in \N$. 

 We conclude by showing that $N$ is conjugate to a subgroup of $\{\ee{b^\ell x+ay^\ell,by}: a\in k, b\in k^*\}:=N_\ell$.

We first show that $H \leq  \{x+ay^\ell:a\in k\}$.
Otherwise, take $\phi=\ee{x+P(y),y}\in H$ with $P(y)=\sum_i a_iy^i$ and $a_i \neq  0$ for some $i \neq  \ell$. Recall that $\pi(N)=G_{y,m}$, and so we can find $\psi=\ee{b^\ell x+Q(y),by}\in N$ with $b^m\neq 1$ for all $m\in\N^{>0}$. Then \[ [\phi,\psi]=x+b^{-\ell}P(by)-P(y)=x+\sum_{i\neq \ell} a_i(b^{-\ell}b^i-1)y^i.\] Thus, $[\phi,\psi]\not\in \{x+ay^\ell:a\in k\}$. Now inductively, we have $[\cdots[[\phi,\psi],\psi],\cdots]\not\in  \{x+ay^\ell:a\in k\}$, which cannot be trivial, contradicting $N$ being nilpotent. Thus, $H \leq  \{x+ay^\ell:a\in k\}\cong \G_a$. Since $\pi(N)= G_{y,m}$, it follows that $N$ is nilpotent of dimension at most 2, hence is abelian. 

Let $\phi:=\ee{b^\ell x+P(y),by}\in N$ and $\psi:=\ee{d^\ell x+Q(y),dy}\in N$ be independent generic elements of $N$. Then since $N$ is abelian, \[x = [\phi,\psi]=x+(bd)^{-\ell}(b^\ell Q(y)-Q(by)+P(dy)-d^{\ell}P(y)),\] and so $b^\ell Q(y)-Q(by)=d^{\ell}P(y)-P(y)$. Say $P(y)=\sum_{i\leq m} a_i^\phi y^i$ and $Q(y)=\sum_{j\leq m} c_j^\psi y^j$. Then for all $i\neq \ell$, we have \[c_i^\psi(b^\ell-b^i)y^i=a_i^\phi(d^\ell-d^i)y^i.\] Hence, $a_i^\phi/(b^\ell-b^i)=c_i^\psi/(d^\ell-d^i)$ (since $\phi$ and $\psi$ are generic, $b^\ell-b^i\neq 0$ and $d^\ell-d^i\neq 0$ for $i\neq \ell$). Since $\phi\ind\psi$, we get $a_i^\phi, b\ind c_i^\psi, d$. Thus, $a_i^\phi/(b^\ell-b^i)=c_i^\psi/(d^\ell-d^i)=:\alpha_i\in\acl(\emptyset)$, and so $\phi$ and $\psi$ belong to \[\{\ee{b^\ell x+cy^\ell+\sum_{i\neq \ell, i\leq m}\alpha_i(b^\ell-b^i)y^i,by}: b\in k^*, c\in k\}=:\tilde N.\] Let $\xi := \langle x+\sum_{i\leq m,i\neq\ell}\alpha_i y^i,y\rangle$. Then it is easy to check that $\tilde N=(N_\ell)^\xi$. Since $\tilde N$ is defined over $\acl(\emptyset)$ and contains the generic $\phi$ of $N$, we get $N \leq  \tilde N$, so $N$ is conjugate to a subgroup of $N_{\ell}$ as desired.
\end{proof}

\section{Correspondence triangles with bounded dimension families of infinite fixed point sets}

We work in an ultrapower $K = \C^\U$ of the complex field, and
we follow the notation and setup of \cite[Section 2]{cubicSurfaces};
recall in particular the definitions of $\bdl$, $\dind$, $\acl^0$, $\trd^0$, $\equiv^0$,
and wgp. We write $\code{Z}$ for the \emph{code} of a constructible set $Z$, 
i.e.\ a finite tuple of generators for the least field over which $Z$ is 
defined.

Consider $G := \Aut(K^2)$ with its action on $X := K^2$.

\begin{definition}
  We call a triple $(a,g,a*g) \in X \times G \times X$ a \defn{generic 
  correspondence triangle} if
  \begin{itemize}
    \item $\trd^0(a)=\dim(X)$;
    \item $a$ and $g$ are wgp;
    \item $a \dind g \dind g*a$.
  \end{itemize}
\end{definition}

\begin{lemma} \label{l:bndFixBdl}
    Suppose $(a,g,g*a) \in X\times G\times X$ is a generic correspondence triangle.

    Let $m \in \N$ be such that for any $h \equiv ^0 g$, if $F_h := 
    \Fix(\gfrac gh)$ is infinite then $\trd^0(\code{F_h}/g) < m$.
        
    Then $\bdl(g) \leq  m\bdl(a)$.
\end{lemma}
\begin{proof}
    Let $t \in \N$ be such that for any $h \equiv ^0 g$,
    if $|F_h| > t$ then $F_h$ is infinite;
    such a $t$ exists by, for example, elimination of $\exists^{\infty}$ in 
    $\ACF_0$.

    Let $\a = (a_0,\ldots ,a_{m-1})$ with $a_i \equiv _g a$ and $a_i \dind_g a_{<i}$.
    Write $g*\a$ for the diagonal action.

    Let $s := t^m+1$. We show that the relation $y=z*x$ on 
    $\tp(\a,g*\a)\times\tp(g)$ omits $K_{s,2}$.

    Indeed, suppose for a contradiction that there is $h \equiv ^0 g$ with $h \neq  g$
    and distinct realisations $\a_0,\ldots ,\a_{t^m}$ of $\tp^0(\a)$
    such that $g*\a_i = h*\a_i$ for all $i$.
    Then by the pigeonhole principle,
    the projections of these tuples to some co-ordinate has cardinality $>t$,
    and hence $|F_h| > t$.

    Let $e := \code{F_h}$.
    By the choices of $t$ and $m$, we have $\trd^0(e/g) < m$.
    Recalling $\a_0 \ind ^0 g$,
    we obtain $\trd^0(\a_0/ge) \geq  \trd^0(\a_0) - \trd^0(e/g) > \trd^0(\a_0) - m$.

    But $h \neq  g$ and the action on the irreducible variety $X$ is faithful and by rational maps,
    so we have $\dim(F_h) < \dim(X)$, but $\a_0 \in F_h^m$,
    and so $\trd^0(\a_0/ge) \leq  m\dim(F_h) \leq  m\dim(X) - m = \trd^0(\a_0) - m < \trd^0(\a_0/ge)$, contradiction.

    Now \cite[Lemma~2.15]{BB-cohMod} applies
    and we conclude as in \cite[Lemma~7.7]{homogES}.
\end{proof}
\begin{remark}
    Note that we can always take $m = \trd^0(g)+1$, but we will be interested in 
    cases where we can do better.
\end{remark}
\begin{remark}
  A natural level of generality for Lemma~\ref{l:bndFixBdl} would be that of 
  an Ind-constructible group $G$ (in the sense of \cite[Definition~7.1]{HPP}) 
  with a faithful birational action on an irreducible variety $X$, all over 
  $\acl^0(\emptyset)$.
\end{remark}

\section{$\Aut(K^2)$}
We continue to consider $G = \Aut(K^2)$ and its action on $X = K^2$.

\begin{lemma} \label{l:bndFix}
    Let $d \in \N_{>0}$, and let $g \in \left<{A, E_d}\right>$.
    Suppose $h \equiv ^0 g$ and
    $F_h := \Fix(\gfrac gh)$ is infinite. Then $\trd^0(\code{F_h}/g) \leq  \max(\dim(A),\dim(E_d))$.
\end{lemma}
\begin{proof}
    First suppose $g \notin S$, so say
    $h=\prod_{i\leq N}h_i$ and $g=\prod_{i\leq N}g_i$ are alternating words with
    $h_i \in \acl^0(h)$ and $g_i \in \acl^0(g)$.
    Note that $h_i \in A \Leftrightarrow g_i \in A$ for all $i$, by 
    Fact~\ref{f:altWords}.
    By Fact~\ref{f:infFix}, $\gfrac gh$ is conjugate to an element of $E$, so 
    by Lemma~\ref{l:conjugate},
    $\gfrac gh = \beta^\eta$ where $\beta \in A\cup E$ and $\eta = 
    \prod_{i>m}g_i$, where $m \leq  N$.
    Note that actually $\beta \in A\cup E_d$,
    since $\beta^\eta, \eta \in \left<{A,E_d}\right>$ and $E \cap \left<{A,E_d}\right> = E_d$.

    In the case $g \in S$, we similarly have $\gfrac gh = \beta^\eta$
    with $\eta=1$ and $\beta = \gfrac gh \in S \subseteq  A\cup E_d$.

    Now $F_h = \eta^{-1}*\Fix(\beta)$,
    so $\code{F_h} \in \acl^0(\eta,\beta) \subseteq  \acl^0(g,\beta)$,
    and so $\trd^0(\code{F_h}/g) \leq  \trd^0(\beta) \leq  \max(\dim(A),\dim(E_d))$.
\end{proof}

\begin{lemma} \label{l:neta}
    Suppose $d \in G$, and $d = \prod_{i<n} f_i$ is an alternating word, and
    \begin{itemize}\item $n>1$;
    \item $\tp(d)$ is wgp;
    \item $f_i \in \acl^0(d)$;
    \item $\code{f_0\cdot S} \notin \acl^0(\emptyset )$
        and $\code{S\cdot f_{n-1}} \notin \acl^0(\emptyset )$.
    \end{itemize}

    Suppose $(d_0,\ldots ,d_{2N-1})$ is a $\dind$-independent sequence of 
    realisations of $\tp(d)$.
    Let $h_N := \prod_{i<2N} d_i^{(-1)^{i+1}}$.
    Then $\bdl(h_N) \geq N\bdl(\code{f_0\cdot S})$.
\end{lemma}
\begin{proof}
    Take $f_{i,j}$ such that $d(f_j)_{j<n} \equiv  d_i(f_{i,j})_{j<n}$.

    Let $c_i := \code{f_{i,0} \cdot S} \in \acl^0(d_i) \setminus \acl^0(\emptyset )$.
    Then $(c_i)_i$ are $\dind$-independent and $\ind ^0$-independent,
    and in particular distinct (since $c_i \notin \acl^0(\emptyset )$),
    so $\gfrac{f_{i+1,0}}{f_{i,0}} \notin S$.
    
    Similarly, $f_{i,n-1}\cdot f_{i+1,n-1}^{-1} \notin S$.

    Hence for $k \leq  N$,
    $h_k = \prod_{i<2k} d_i^{(-1)^{i+1}}$ can be expressed as an alternating word as follows:
    $$h_k = f_{0,n-1}^{-1}\cdot\ldots \cdot f_{0,1}^{-1}\cdot(\gfrac{f_{1,0}}{f_{0,0}}) \cdot f_{1,1} \cdot \ldots \cdot f_{1,n-2}\cdot (f_{1,n-1}\cdot f_{2,n-1}^{-1})\cdot f_{2,n-2}^{-1}\cdot \ldots  \cdot f_{2k-1,n-1} .$$

    Let $c_{2k+1}' := \code{h_k\cdot d_{2k}^{-1} \cdot f_{2k+1,0} \cdot S}$.
    By Lemma~\ref{l:algPrefix}, $c_{2k+1}' \in \acl^0(h_l)$ if $l > k$.
    But $c_{2k+1}'$ is interalgebraic with $c_{2k+1}$ over $d_{\leq 2k}$.
    So
    \begin{align*} \bdl(h_N)
    &\geq  \bdl((c_{2k+1}')_{k<N}) \\
    &= \sum_{k<N} \bdl(c_{2k+1}'/(c_{2i+1}')_{i<k}) \\
    &\geq  \sum_{k<N} \bdl(c_{2k+1}'/d_{\leq 2k}) \;\;\text{(since $c_{2i+1}' \in \acl^0(h_k) \subseteq  \acl^0(d_{\leq 2k})$)} \\
    &= \sum_{k<N} \bdl(c_{2k+1}/d_{\leq 2k}) \;\;\text{(since $c_{2k+1} \in \acl^0(d_{2k+1})$ and $d_{2k+1}\dind d_{\leq 2k}$)}\\
    &= \sum_{k<N} \bdl(c_{2k+1}) \\
    &= N \bdl(c_0) ,\end{align*}
    as required.
\end{proof}

\begin{lemma}\label{l:reduceToNilp}
    Suppose $(a,g,g*a) \in X\times G\times X$ is a generic correspondence triangle.
    Suppose $Y := \locus^0_G(g)$ is contained in a left coset of an algebraic subgroup $H \leq G$.
    Then $Y$ is contained in a left coset $C$ of a connected nilpotent algebraic subgroup of 
    $H$, with $C$ defined over $\acl^0(\emptyset)$.
\end{lemma}
\begin{proof}
  Let $N := \langle Y^{-1}Y \rangle \leq H$.
  Since $G$ acts faithfully, no non-trivial element of $N$ fixes $X = \locus^0(a)$ pointwise.
  By \cite[Theorem~A.4]{cubicSurfaces} (which is based on results in \cite{BGT-lin}), $N$ is nilpotent as required.
  Finally, $N$ is over $\acl^0(\emptyset)$ by definition, and hence so is $C := 
  YN$.
\end{proof}

\begin{theorem} \label{t:main}
    Suppose $(a,g,g*a) \in X\times G\times X$ is a generic correspondence triangle.
    Then $g \in C$ for some left coset $C$ of a conjugate of a connected nilpotent algebraic subgroup $N$ of $A\cup E$,
    with $C$ defined over $\acl^0(\emptyset)$.
    
    Hence, by Lemmas~\ref{l:nilS} and \ref{l:nilE}, $N$ can to be taken to be one of the nilpotent groups listed in Lemma~\ref{l:nilE}.
\end{theorem}
\begin{proof}
    By Lemma~\ref{l:reduceToNilp}, it suffices to show the conclusion without ``nilpotent''.

    If $g \in E\cup A$, we are done, since $A$ and each $E_d$ is an algebraic subgroup.
    So say $g = \prod_{i<n} f_i$ is an alternating word with $n>1$.

    Let $c := \code{f_0\cdot S}$ and $c' := \code{S\cdot f_{n-1}}$.
    If $c \in \acl^0(\emptyset )$,
    then say $f_0' \in f_0\cdot S \cap \acl^0(\emptyset )$,
    and let $g' := \gfrac g{f_0'}$.
    Then $g'$ can be written with a shorter alternating word than $g$.

    Since the property of being a left coset of a conjugate of an algebraic 
    subgroup
    of $E\cup A$ is preserved by multiplication both on the left and on the right,
    iterating this process,
    we may assume that $c,c' \notin \acl^0(\emptyset )$.
    Note then that $c \nind ^0 g$, so $\bdl(c)>0$ by wgp.

    Inductively define
    \begin{itemize}\item $g_0 := g$;
    \item $g_i'$ such that $g_i' \equiv _{g_i*a} g_i$ and $g_i' \dind_{g_i*a} g_i$;
    \item $g_{i+1} := \gfrac{g_i}{g_i'}$.
    \end{itemize}

    We show by induction $i$ that $(a,g_i,g_i*a)$ is a generic correspondence triangle.
    This holds for $i=0$ by assumption. Suppose it holds for $i$.
    Then $g_i' \dind g_i*a$, so $g_i' \dind (g_i*a)g_i$, and so $g_i' \dind ag_i$,
    hence $a \dind g_ig_i'$ and in particular $a\dind g_{i+1}$.
    Also $g_i'(g_{i+1}*a) \equiv _{g_i*a} g_ia$, since $a = g_i^{-1}*g_i*a$ and $g_{i+1}*a = (g_i')^{-1}*g_i*a$,
    so $g_{i+1}*a \dind g_i'$,
    but as above $g_i \dind (g_i*a)g_i'$ and so $g_i \dind (g_{i+1}*a)g_i'$,
    hence $g_{i+1}*a \dind g_ig_i'$ and in particular $g_{i+1}*a \dind g_{i+1}$.
    Finally, $g_{i+1}$ is wgp by \cite[Lemma~2.13]{cubicSurfaces}.

    Say $d$ is such that $g \in \left<{A,E_d}\right>$, and let $D := \max(\dim(A),\dim(E_d))$.
    Then by Lemma~\ref{l:bndFix}
    and Lemma~\ref{l:bndFixBdl} applied to the generic correspondence triangle $(a,g_i,g_i*a)$,
    we have $\bdl(g_i) \leq  (D+1)\bdl(a)$.

    However, by Lemma~\ref{l:neta}, we have $\bdl(g_{i+1}) \geq 2^i\bdl(c)$,
    so for $i$ large enough this contradicts $\bdl(c)>0$.
\end{proof}

\begin{remark}\label{r:genAmalg}
    This proof would go through generally for an amalgamated product $G = G_1 
    \star_{G_0} G_2$ of Ind-constructible groups acting faithfully 
    birationally on an irreducible variety $X$,
    where $G_0$ is constructible, $G_1$ and $G_2$ are unions of constructible subgroups, 
    and each $g \in G \setminus (G_1\cup G_2)^G$ has $\Fix(g)$ finite. However, this 
    is still a rather special situation, which for example does not include 
    the case of the Cremona group of birational maps of the plane.
\end{remark}

\section{Finitary consequences}
With the same proof as \cite[Theorem 9.3]{homogES}, we get the following finitary consequence.

\begin{definition}
    Let $F\subseteq \Aut(\C^2)$ be a finite subset and $\varepsilon>0$. We say $F$ is \emph{$\varepsilon$-nilpotent}, if there are $f,g\in \Aut(\C^2)$ and a connected nilpotent algebraic subgroup $N_d\leq E_d$ such that \[|F\cap fN_dg |\geq |F|^{1-\epsilon}.\]
\end{definition}

\begin{corollary}\label{c:main-fin}
For all $\varepsilon > 0$ and $n\in\N$, there is $\eta>0$ such that the following holds. 
    Let $F\subseteq \Aut(\C^2)$ be a finite set of polynomial automorphisms consisting of polynomials of degree at most $n$,
    and let $A$  be a finite subset of $\C^2$. Suppose
    \begin{itemize}
        \item $|A|^{1/\varepsilon}\geq |F|\geq |A|^{\varepsilon}\geq \frac{1}{\eta}$;
        \item For all algebraic curves $C \subseteq \A^2$ of complexity $<\frac1\eta$, we have $|A \cap C| \leq |A|^{1-\varepsilon}$;
        \item $F$ is not $\varepsilon$-nilpotent.
    \end{itemize}
    Then \[|F*A| = |\{f*a:f\in F,a\in A\}| \geq |A|^{1+\eta}.\]
\end{corollary}


Combining these techniques with our analysis of the nilpotent algebraic subgroups of $\Aut(\C^2)$, we can now prove the statement in the spirit of Elekes-Rónyai given in the introduction.

\begin{proof}[Proof of Theorem~\ref{t:ER-main}]
    Suppose for a contradiction that there is some $\varepsilon > 0$ with no $\eta$ satisfying the conclusion of the theorem. Taking $\eta=1/n$ we can find $A_n\subseteq \C^2,B_n\subseteq \C^{|\bar{z}|}$ such that \begin{enumerate}
        \item $|A_n|^{1/\varepsilon}\geq |B_n|\geq |A|^\varepsilon\geq n$;
        \item $(F(x,y,b),G(x,y,b))\in\Aut(\C^2)$ for all $b\in B_n$;
        \item $|A_n\cap V|\leq |A_n|^{1-\varepsilon}$ and $|B_n\cap W|\leq |B_n|^{1-\varepsilon}$ for all varieties $V\subsetneq \C^2$, $W\subsetneq \C^{|\bar z|}$ of complexity $<n$;
        \item $|(F,G)(B_n)*A_n| < |A_n|^{1+1/n}$
    \end{enumerate}
    Work in the ultraproduct $K=\C^\mathcal{U}$, and assume $F$ and $G$ are defined over $k_0 := \acl^0(\emptyset)$. Let $\xi:=(|A_n|)_n\in\N^\mathcal{U}$ let $\bdl:=\bdl_\xi$.
    Let $A:=\prod_{n\in\mathcal{U}} A_n$,  $B:=\prod_{n\in\mathcal{U}} B_n$ and $C:=(F,G)(B)*A$. Then by (1), we have $\bdl(B)\in [\varepsilon,1/\varepsilon]$, and  $\bdl(C)=\bdl(A)=1$ by (4). Let $(a,b)\in A\times B$ be such that $\bdl(a,b)=\bdl(A)+\bdl(B)$. Then $a\dind b$, $\bdl(a)=\bdl(A)$ and $\bdl(b)=\bdl(B)$. Let $f_b=(F(x,y,b),G(x,y,b))$. Then $f_b\in\Aut(K^2)$ by (2), and $f_b \in \acl^0(b)$ by choice of $k_0$, so $a\dind f_b$. By (3) we also have $(a,b)$ is generic in $K^{|\bar z|+2}$ and $\tp(a),\tp(b)$ wgp, therefore $\tp(f_b)$ is also wgp. If $\bdl(f_b)=0$, then $f_b\in\acl^0(\emptyset)=k_0$, namely $F(x,y,b)=f_0(x,y)$ and $G(x,y,b)=g_0(x,y)$ for some $f_0,g_0\in k_0[x,y]$. Since $b$ is generic in $K^{|\bar z|}$ over $k_0$, we get the equality of polynomials $F(x,y,\bar z)=f_0(x,y)$, contradicting $F\not\in\C[x,y]$. Thus $\tp(f_b)$ is broad. Consider $f_b*a\in C$, we must have $\bdl(f_b*a)\leq \bdl(C)=\bdl(A)=\bdl(a)$. Thus, $\bdl(f_b)+\bdl(f_b*a)\geq\bdl(f_b,f_b*a)=\bdl(f_b,a)=\bdl(f_b)+\bdl(a)\geq\bdl(f_b)+\bdl(f_b*a)$ and we must have the equality, namely $f_b*a\dind f_b$. In conclusion, $(a,f_b,f_b*a)$ is a generic correspondence triangle. By Theorem~\ref{t:main}, there is a connected nilpotent algebraic group $N$ from one of the four possibilities listed in Lemma~\ref{l:nilE} and $g,h\in\Aut(K^2)$ with $N,h,g$ all defined over $k_0=\acl^0(\emptyset)$ such that $gf_bh\in N$. We treat the case that $N$ is conjugate to $\{(x,y)\mapsto(cx,dy):c,d\in K^*\}=G_{x,m}\times G_{y,m}$; the other cases are similar. Since $N$ is defined over $k_0$, by multiplying $g$ and $h$ with elements in the automorphism group $\Aut(k_0^2)$, we may assume $gf_bh\in G_{x,m}\times G_{y,m}$. Now $(cx,dy) = g\circ (F(x,y,b),G(x,y,b))\circ h = (\sum_{i,j} t_{i,j}(b)x^iy^j, \sum_{i,j} t'_{i,j}(b)x^iy^j)$ with $t_{i,j},t'_{i,j} \in k_0[\bar z]$. Thus $c = t_{1,0}(b)$ and $d = t'_{0,1}(b)$ and $t_{i,j}(b)=0=t'_{i,j}(b)$ for all other $i,j$. Now $b$ is generic over $k_0$, so we get the equality $g\circ (F(x,y,\bar z),G(x,y,\bar z))\circ h = (t_{1,0}(\bar z)x, t'_{0,1}(\bar z)y)$ in $k_0[\bar z][x,y]$. So $(F,G)$ is co-ordinate separable, contrary to assumption.
\end{proof}

\bibliographystyle{alpha}
\bibliography{autC2}
\end{document}